\documentclass[11pt, english]{amsart}

\usepackage[utf8x]{inputenc}

\usepackage[T1]{fontenc}

\usepackage{libertine}

\usepackage[style=british]{csquotes}

\usepackage{epigraph}
\usepackage{xcolor}
\colorlet{darkteal}{teal!70!black}

\DeclareFontFamily{U}{mathx}{\hyphenchar\font45}
\DeclareFontShape{U}{mathx}{m}{n}{
      <5> <6> <7> <8> <9> <10>
      <10.95> <12> <14.4> <17.28> <20.74> <24.88>
      mathx10
      }{}
\DeclareSymbolFont{mathx}{U}{mathx}{m}{n}
\DeclareFontSubstitution{U}{mathx}{m}{n}
\DeclareMathAccent{\widecheck}{0}{mathx}{"71}

\usepackage[naturalnames=true, pageanchor=true, hyperfootnotes=false, breaklinks, colorlinks, allcolors=darkteal]{hyperref}

\theoremstyle{plain}
\newtheorem{proposition}{Proposition}
\newtheorem{theorem}[proposition]{Theorem}
\newtheorem{corollary}[proposition]{Corollary}
\newtheorem{lemma}[proposition]{Lemma}
\theoremstyle{definition}
\newtheorem{definition}[proposition]{Definition}
\newtheorem{remark}[proposition]{Remark}
\newtheorem{question}[proposition]{Question}
\newtheorem*{convention}{Convention}

\DeclareMathOperator{\Homeo}{Homeo} 
\DeclareMathOperator{\Sym}{Sym} 
\DeclareMathOperator{\Alt}{Alt} 
\DeclareMathOperator{\Aut}{Aut} 
\DeclareMathOperator{\Fixa}{Fix} 
\DeclareMathOperator{\LO}{LO} 
\DeclareMathOperator{\id}{id} 

\newcommand{\Nat}{\mathbf{N}}
\renewcommand{\phi}{\varphi}
\newcommand{\Ti}[1][i]{\mathbf{T}_{#1}}
\newcommand{\compl}[1]{\widehat{#1}}
\newcommand{\ucompl}[1]{\widecheck{#1}} 
\newcommand{\Rcompl}[1]{\widehat{#1}} 
\newcommand{\umf}[1]{\mathrm{UMF}\left(#1\right)}

\newcommand{\Power}{\mathcal{P}}
\newcommand{\SC}{\beta}
\newcommand{\CB}{\mathrm{CB}}
\newcommand{\Graph}{\mathcal{G}}
\newcommand{\Tree}{\mathcal{T}}
\newcommand{\chain}{\epsilon}

\newcommand{\no}{n°}
\newcommand{\Raikov}{Ra\u{\i}kov}

\begin{document}

\vspace*{-2em}

\title[Groups of homeomorphisms of scattered spaces]{Dynamics and structure\\ of groups of homeomorphisms\\ of scattered spaces}
\author{Maxime Gheysens}
\address{Institut für Diskrete Mathematik und Algebra, Technische Universität Bergakademie Freiberg, 09596 Freiberg, Germany}
\email{maxime.gheysens@math.tu-freiberg.de}
\thanks{Work supported in part by the European Research Council Consolidator Grant no.~681207.}
\date{30 November 2020}

\begin{abstract}
	We study the topological structure and the topological dynamics of groups of homeomorphisms of scattered spaces. For a large class of them (including the homeomorphism group of any ordinal space or of any locally compact scattered space), we establish Roelcke-precompactness and amenability, classify all closed normal subgroups and compute the universal minimal flow. As a by-product, we classify up to isomorphism the homeomorphism groups of compact ordinal spaces. 
\end{abstract}

\maketitle

\tableofcontents
\clearpage

A topological space is said \emph{scattered} if it does not contain any nonempty perfect subset. As we observed in \cite{Gheysens_omega}, this \enquote{discrete feature} allows the topology of pointwise convergence on a scattered space $X$ to be compatible with the group structure of the group of homeomorphisms $\Homeo(X)$. Moreover, the canonical map $\Homeo(X) \hookrightarrow \Sym(X)$ is then a topological embedding (when $\Sym(X)$ is endowed with the topology of pointwise convergence on the \emph{discrete} set $X$).

Except in trivial cases, the subgroup $G = \Homeo(X)$ of $\Sym(X)$ is not closed (and, in particular, is not complete for the upper uniform structure). We show here how, in most cases of interest, including all locally compact scattered spaces, we can very easily compute its closure $\overline{G}$, which happens to be a product of symmetric groups.

The point of this seemingly technical result is that many topological and dynamical properties of a topological group $G$ can be studied indifferently on $G$ or on its closure $\overline{G}$ for any topological embedding into another group $H$. (More abstractly, they can be studied on any group lying between $G$ and its upper completion $\ucompl{G}$.) This is notably the case of continuous actions on compact spaces. In particular, we will establish amenability and compute the universal minimal flow of $\Homeo(X)$ when $X$ is a locally compact scattered space (Corollary~\ref{cor:loccompscatt}). Other properties we will study that way are Roelcke-precompactness and topological perfectness (of the group) --- all notions are recalled in Section~\ref{sec:backgrounds} below.

The general scheme for all the proofs of this paper is the following three-step argument:
\begin{enumerate}
	\item When a scattered space $X$ is zero-dimensional (a property that holds automatically in the locally compact case), the canonical action of the homeomorphism group $\Homeo(X)$ on $X$ enjoys a strong \emph{transitivity} property.
	\item This transitivity property allows to describe explicitly and easily the upper completion of $\Homeo(X)$ (or, equivalently, its closure in $\Sym(X)$) as a product of symmetric groups.
	\item The properties we study (amenability, Roelcke-precompactness, and so on) are inherited by (and from) dense subgroups and are easy to establish for a product of symmetric groups.
\end{enumerate}

Instead of writing each proof along those lines, we organize the whole paper according to that scheme. We start by the latter step and collect in Section~\ref{sec:densecompl} all the needed facts about dense subgroups and product of groups. In Section~\ref{sec:fullytransitive}, we define \emph{fully transitive} groups of homeomorphisms and study for them amenability, universal minimal flows, Roelcke-precompactness, and so on. Lastly, in Section~\ref{sec:examples}, we prove that the homeomorphism group of a zero-dimensional scattered space is fully transitive. In order to give more context to this result, we also give in Section~\ref{sec:exotic} some examples of homeomorphism groups of scattered spaces that fail to be amenable or Roelcke-precompact, thus showing that we cannot get rid of the \enquote{zero-dimensional} assumption. Finally, as another application of full transitivity, we conclude the paper by classifying up to isomorphism the homeomorphism group of compact ordinal spaces (Section~\ref{sec:classif}).

\begin{remark}
	This paper is a companion to \cite{Gheysens_omega}, but the two papers differ in scope and style and can be read independently. Whereas \cite{Gheysens_omega} focuses on a particular group of interest, $\Homeo(\omega_1)$, and leverages the peculiar structure of the space $\omega_1$ to give direct ad hoc proofs, the present paper aims at a more complete picture of the class of spaces the homeomorphism group of which enjoys similar topological and dynamical properties. Inevitably lost in this generality are all the \enquote{uncountable features} of $\Homeo(\omega_1)$ (Sections~2.1 and 2.2 of \cite{Gheysens_omega}).
\end{remark}

\begin{convention}
	In this paper, the groups $\Sym(X)$ are always endowed with the topology of pointwise convergence on the set $X$ endowed with the discrete topology. If $X$ is a scattered space, the group $\Homeo(X)$ is endowed with the topology of pointwise convergence on $X$ (for which it is a topological group), otherwise, it is simply considered as an abstract group.
\end{convention}

\section{Backgrounds}\label{sec:backgrounds}

\subsection{On scattered spaces}

A topological space $X$ is said \emph{scattered} if any nonempty subset $A$ of $X$ contains an isolated point (for its subspace topology). In other words, $X$ does not contain any nonempty perfect subspace. Notable (non-discrete) examples include ordinal spaces, that is, ordinals endowed with their order topology (on which we will focus in Section~\ref{sec:classif}). When $X$ is scattered, the topology of pointwise convergence on $X$ is compatible with the group structure of $\Homeo(X)$ \cite[Cor.~2]{Gheysens_omega}.

A \enquote{constructive} approach to scattered spaces is given by the Cantor--Bendixson derivative process. For a topological space $X$, we define its \emph{derived subspace}, $X'$, as the subspace of all limit points of $X$. By transfinite induction, we can now define a nested sequence of subspaces of $X$ by $X^{(0)} = X$ ; $X^{(\alpha + 1)} = \left(X^{(\alpha)}\right)'$ and $X^{(\lambda)} = \bigcap_{\alpha < \lambda} X^{(\alpha)}$ (for $\lambda$ a limit ordinal). This sequence stabilises at some ordinal $\CB(X)$, the \emph{Cantor--Bendixson rank of $X$}. In particular, we can define for any point $x$ its \emph{Cantor--Bendixson rank} as the greatest ordinal $\alpha \leq \CB(X)$ such that $x \in X^{(\alpha)}$. The subspace $X^{(\CB(X))}$ is the largest closed perfect subspace of $X$ (sometimes called the \emph{perfect kernel} of $X$); by definition, a space $X$ is scattered if and only if $X^{(\CB(X))} = \emptyset$.

Before recalling further topological notions, let us stress that our topological spaces \emph{are not by default assumed to be Hausdorff}\footnote{On the other hand, topological \emph{groups} in this paper happen to be Hausdorff.}. The reason for this liberality is that, when $X$ is a scattered space, the topological group $\Homeo(X)$ is \emph{always} Hausdorff even if $X$ fails to be so. This can be seen in two ways.
\begin{enumerate}
	\item If $X$ is scattered, then its subset of isolated points is dense. Therefore, the intersection of all identity neighbourhoods, which is contained in the set of all $h$ that fix each isolated point, contains only the identity. Thus $\Homeo(X)$ is $\Ti[1]$, hence Hausdorff since it is a topological group.
	\item More generally, for any topological space $X$ and set $Y$, the pointwise convergence topology (i.e.~the product topology) on $X^Y$ is $\Ti$ when $X$ is so ($i = 0, 1, 2$). But a scattered space is always at least $\Ti[0]$ (see below). Hence $\Homeo(X)$ (a subset of $X^X$) is also $\Ti[0]$ when $X$ is scattered, hence is Hausdorff since it is a topological group.
\end{enumerate}

Since we do consider (mostly in Section~\ref{sec:exotic}) non-Hausdorff spaces, we recall for the convenience of the reader a few facts and terms useful in that generality.
\begin{enumerate}
	\item A topological space $X$ is said $\Ti[0]$ if distinct points do not belong to the same open sets: for any $x \neq y$, there is an open set containing $x$ but not $y$ \emph{or} containing $y$ but not $x$. It is said $\Ti[1]$ if distinct points can be separated by distinct open sets: for any $x \neq y$, there is an open set containing $x$ but not $y$. (Equivalently, if singletons are closed.) It is said $\Ti[2]$ or \emph{Hausdorff} if distinct points can be separated by disjoint open sets: for any $x \neq y$, there exist disjoint open sets $U$ and $V$ such that $x \in U$ and $y \in V$.
	\item A scattered space is always $\Ti[0]$. Indeed, for any two distinct points $x, y \in X$, at least one of them, say $x$, is an isolated point of the subspace $\{x, y\}$. This means that there is an open set in $X$ containing $x$ but not $y$.
	\item A topological space is said \emph{totally disconnected}\footnote{\enquote{Hereditarily disconnected} in the terminology of \cite{Engelking_GT}.} if its only nonempty connected subspaces are singletons.
	\item A scattered $\Ti[1]$ space is always totally disconnected. Indeed, let $C$ be a nonempty connected subspace of $X$. Since $X$ is scattered, there must exist an isolated point $x \in C$. This means that $\{x\}$ is open in $C$. But since $X$ is $\Ti[1]$, singletons are always closed. Therefore, by connectedness, $C = \{x\}$.
	\item A topological space is said \emph{zero-dimensional} if its topology admits a basis of clopen\footnote{A set is \emph{clopen} if it is both closed and open.} sets. A $\Ti[0]$ zero-dimensional space is automatically Hausdorff: indeed, if $U$ is an open set containing $x$ but not $y$, then there exists a clopen set $V$ in $U$ containing $x$, hence $V$ and its complement are disjoint neighbourhoods of $x$ and $y$, respectively\footnote{Because of this fact and of the weakness of the $\Ti[0]$ axiom, some authors already include a separation axiom in the definition of zero-dimensional, cf.~\cite[\textsc{ix}, \S~6, \no~4, déf.~6]{Bourbaki_TG_V} or \cite[6.2]{Engelking_GT}.}. Consequently, a zero-dimensional scattered space is necessarily Hausdorff.
    \item We fill follow Bourbaki's terminology by defining \emph{compact} to mean \enquote{Hausdorff and quasi-compact} and \emph{locally compact} to mean \enquote{Hausdorff and locally quasi-compact}, where a space is said \emph{locally quasi-compact} if each point admits a quasi-compact neighbourhood.
    \item Of cardinal importance for our applications is the following theorem of van Dantzig: a locally compact totally disconnected space is always zero-dimensional (cf.~e.g.~\cite[\textsc{ii}, \S~4, \no~4, cor.]{Bourbaki_TG_I} or \cite[6.2.9]{Engelking_GT}).
\end{enumerate}

For the sake of completeness, we give here references to the counterexamples that are expected from the above implications.

\begin{enumerate}
	\item The Hausdorff assumption is essential in van Dantzig's theorem: a $\Ti[1]$ locally quasi-compact and totally disconnected space need not be zero-dimensional, cf.~Example~99 in \cite{SS_1978} (observe that this example is also scattered).
	\item A scattered $\Ti[0]$ space need not be $\Ti[1]$ (cf.~Examples~8--12 in \cite{SS_1978}) and a scattered $\Ti[1]$ space need not be $\Ti[2]$ (cf.~again Example~99 in \cite{SS_1978}).
\end{enumerate}

\subsection{On uniform structures on groups}

Let $G$ be a topological group. We recall that there are four natural uniform structures definable on a topological group $G$, the \emph{left, right, upper} and \emph{lower} (or \emph{Roelcke}) uniform structures, see e.g.~\cite[chap.~2]{RD_1981}. They are defined respectively by the following basis of entourages:
\begin{description}
	\item[left] $\left\{(x, y)\ \middle|\ x \in yU\right\}$,
	\item[right] $\left\{(x, y)\ \middle|\ x \in Uy\right\}$,
	\item[upper] $\left\{(x, y)\ \middle|\ x \in Uy \cap yU\right\}$,
	\item[lower] $\left\{(x, y)\ \middle|\ x \in UyU\right\}$,
\end{description}
where $U$ ranges among a basis of identity neighbourhoods of $G$.

Only the last two structures will play a crucial rôle in this paper, via two complementary notions, precompactness and completion.

A group $G$ is said \emph{Roelcke-precompact} if it is precompact for the lower uniform structure. In other words, for any identity neighbourhood $U$, there is a finite set $F$ such that $G = UFU$. This precompactness condition ensures some rigidity of the actions of the group (for instance, any isometric continuous actions on a metric space has bounded orbits, see e.g.~\cite[Prop.~2.3]{Tsankov_2012}) while simultaneously being quite common among (non-locally compact) groups. (By contrast, if the group is Hausdorff, the precompactness condition for any of the other three uniform structures boils down to being a dense subgroup of a compact group \cite[9.12 and 10.12(c)]{RD_1981}.) Note that this notion is useless for locally compact groups, which are automatically complete for the Roelcke uniform structure (hence, for them, Roelcke-precompactness is equivalent to compactness) \cite[8.8]{RD_1981}.

Another notion worth of study for these uniform structures is their \emph{completion}.  A distinguishing feature of the upper uniform structure is that its (Hausdorff) completion $\ucompl{G}$ is again a topological \emph{group}. (By contrast, the completion of the right or left uniform structures on a topological group is only a topological semigroup and the completion of the lower uniform structure can even fail to be a semitopological semigroup\footnote{But see \cite[\S~7]{Zielinski_2019} for positive results in the case of locally Roelcke-precompact Polish groups.}). In particular, any topological group is a dense subgroup of a group which is complete for the upper uniform structure. For these facts, see e.g.~\cite[10.12 and 10.24]{RD_1981}. A group is called \emph{\Raikov-complete} if it is complete for its upper uniform structure.

Observe that, if $H$ is a subgroup of a topological group $G$, endowed with its induced topology, then the restriction to $H$ of the upper uniform structure of $G$ is nothing but the upper uniform structure of $H$. (This would also hold for the right and the left uniform structures, but not in general for the lower one, see e.g.~\cite[3.24 and 3.25]{RD_1981}.) In particular, we see that, in order to compute the upper completion of a group, it is sufficient to embed it in a \Raikov-complete group and to compute its closure there.

\begin{remark}
	The right uniform structure has a cameo rôle in this paper, via its link to amenability and extreme amenability, as explained in the next subsection.
\end{remark}

\subsection{On amenability and universal minimal flows}

A topological group is said \emph{amenable} if, whenever it acts continuously and affinely on a nonempty compact convex set, it has a fixed point. It is said \emph{extremely amenable} if, whenever it acts continuously on a nonempty compact set, it has a fixed point. Here, the continuity of the action of a group $G$ on a space $K$ means the (joint) continuity of the action map $G \times K \rightarrow K$. This is equivalent to requiring the morphism $\phi\colon G \rightarrow \Homeo(K)$ to be continuous when the latter is endowed with the topology of \emph{uniform} convergence (or compact-open topology).

How much a group fails to be extremely amenable is assessed by its \emph{universal minimal flow}, defined as follows. A \emph{flow} is a nonempty compact space $K$ endowed with a continuous action of a topological group $G$. It is said \emph{minimal} if any point in it has a dense orbit; by a standard Zorn argument, any flow contains a minimal subflow. It is quite easy to show that for any group $G$, there is a \enquote{maximal} minimal flow $K$, i.e.~such that there is a $G$-equivariant map from $K$ onto any other minimal flow. Trickier but still true is the fact that this minimal flow is unique up to isomorphism, that is, is universal. It is called the \emph{universal minimal flow} of $G$, written $\umf{G}$. Standard textbooks for all these facts are e.g.~\cite{deVries_1993, ellis_1969}. Observe that a group is extremely amenable if and only if its universal minimal flow is a singleton.

(Adding convexity everywhere, these notions can be adapted to amenability in order to yield the \emph{affine univeral minimal flow} of a topological group. We will not need it in this paper since we will prove that our groups are amenable, that is, their affine universal minimal flow is a singleton.)

If $G$ is a compact group, then $\umf{G} = G$, but if $G$ is a locally compact non-compact group, then $\umf{G}$ is not metrisable \cite[App.~2]{KPT_2005}. Describing explicitly the universal minimal flow of a group is often an ambitious endeavour but here is a particular case where it is tractable. Assume that $G$ is Hausdorff and contains a closed subgroup $H$ which is \emph{extremely amenable} and \emph{co-precompact}. The latter condition means that the set of right classes $G/H$, endowed with the quotient of the right uniform structure of $G$ (cf.~\cite[5.21]{RD_1981}), is precompact. Observe that, for any flow $K$, the orbit map $G \rightarrow K\colon g \mapsto gx$ is equivariant and right-uniformly continuous for any point $x$ in $K$. Chosing $x$ to be a point fixed by $H$, which must exist since $H$ is extremely amenable, we see that this map factors through $G/H$, hence extends to the completion $\compl{G/H}$. Therefore, the space $\compl{G/H}$ admits an equivariant map to any flow of $G$. This means that any \emph{minimal} subflow of $\compl{G/H}$ is isomorphic to the universal minimal flow of $G$ (see \cite[6.2]{Pestov_2006} and in particular its Theorem~6.2.9).

In many cases, there exists an extremely amenable co-precompact subgroup $H$ such that the flow $\compl{G/H}$ is already minimal. In this situation, the universal minimal flow of a product of groups can be easily computed:

\begin{proposition}\label{prop:umfproduct}
	Let $\{G_i\}$ be a family of Hausdorff topological groups. Assume that each $G_i$ contains a closed extremely amenable co-precompact subgroup $H_i$ such that $\umf{G_i} = \compl{G_i/H_i}$. Then the universal minimal flows commute with the product:
\begin{equation*}
    \umf{\prod_i G_i} = \prod_i \umf{G_i}.
\end{equation*}
\end{proposition}

\begin{proof}
	Let us write $H = \prod_i H_i$ and $G = \prod_i G_i$. Since the product of right uniform structures is the right uniform structure of the product (cf.~\cite[3.35]{RD_1981}), the subgroup $H$ is co-precompact in $G$. Moreover, completion also commutes with the product \cite[\textsc{ii}, \S~3, \no~9]{Bourbaki_TG_I}. Hence the compact $G$-space $\compl{G / H}$ is isomorphic to $\prod_i \umf{G_i}$ and therefore minimal.

	Now extreme amenability is preserved by products (see for instance \cite[Lemma~15]{Gheysens_omega} for the analogous statement for amenability --- the proof does not appeal in an essential way to convexity, hence also holds for extreme amenability), so $H$ is also extremely amenable. By the above discussion, this shows that $\compl{G/H}$ admits equivariant maps to any flow of $G$. Hence $\compl{G/H} = \umf{G}$.
\end{proof}

\begin{remark}
	By deep results of Ben Yaacov, Melleray, Nguyen Van The et Tsankov \cite{MNVTT_2016, BYMT_2017}, the hypothesis \enquote{there exists an extremely amenable co-precompact subgroup $H_i$ such that $\umf{G_i} = \compl{G_i / H_i}$} holds in particular whenever $G_i$ is a Polish group with a \emph{metrisable} universal minimal flow. It also holds for the symmetric groups $\Sym(X)$ (see the next subsection).
\end{remark}

\begin{remark}
	Moreover, we point out that Basso and Zucker have recently provided in \cite{BZ_2020} a refined \emph{characterization}  of the commutation of the universal minimal flow with the product operator, via the topology of flows. See in particular Th~6.1, Prop.~7.7 and Cor.~7.8 of \cite{BZ_2020}. The hypothesis of our Proposition~\ref{prop:umfproduct} implies in particular that the groups $G_i$ fit into their general framework (Prop.~8.4 therein).
\end{remark}

\subsection{On symmetric groups}\label{sec:sym}

We collect here a few facts about the symmetric group $\Sym(X)$ of all the bijections of a set $X$ (of any cardinality). We recall that we endow this group with the topology of pointwise convergence on the discrete set $X$.

\begin{proposition}\label{prop:symamenrprec}
	The group $\Sym(X)$ is amenable, Roelcke-precompact and \Raikov-complete.
\end{proposition}

\begin{proof}
	For amenability, observe indeed that the subgroup $H$ of finitely supported permutations of $X$ is dense in $\Sym(X)$ and that $H$ is amenable (even as a discrete group) since it is locally finite. For Roelcke-precompactness, see for instance \cite[9.14]{RD_1981}. \Raikov-completeness can be easily shown by hand or alternativelty be deduced from more general results, such as \Raikov-completeness for the compact-open topology on $\Homeo(X)$ when $X$ is a locally compact space \cite[Th.~6]{Arens_1946} or for automorphism groups of uniformly locally bounded spaces \cite[\S~2.3]{Gheysens_ulb}.
\end{proof}

Moreover, the universal minimal flow of $\Sym(X)$ can be described explicitly as $\LO(X)$, the space of linear orders on the set $X$ (see \cite[\S~2]{GW_2002} for $X$ countable and e.g.~\cite[Th.~4]{Bartosova_2013} for any $X$). We recall that linear orders, being binary relations, can be identified with their characteristic functions on $X \times X$, so that the set of all linear orders can be viewed as a subspace of $2^{X \times X}$, from which it inherits its compact topology and an action by $\Sym(X)$. In other words, a typical neighbourhood of a linear order $<$ is made of all linear orders $<'$ such that $x_i < y_i \Leftrightarrow x_i <' y_i$ for some finite family of couples $(x_1, y_1), \dots, (x_n, y_n) \in X \times X$ and, if $g \in \Sym(X)$, the order $\mathbin{(g\!<)}$ is defined by $x \mathbin{(g\!<)} y \Leftrightarrow g^{-1} x < g^{-1} y$ for any $x, y \in X$. 

Let us note that this description of the universal minimal flow of $\Sym(X)$ fits in the particular case explained above:

\begin{proposition}\label{prop:umfsym}
	Let $X$ be any set and $G = \Sym(X)$. There exists a linear order $<$ on $X$ such that $\Aut(X, <)$ is a closed extremely amenable co-precompact subgroup of $G$ and that $\LO(X)$, the universal minimal flow of $G$, is the completion of the quotient $G/\Aut(X, <)$.
\end{proposition}

\begin{proof}[References for the proof] The result is straightforward when $X$ is finite. Indeed, in that case, $\Sym(X)$ is finite, hence equal to its universal minimal flow. Moreover, any finite linear order $<$ is a well-ordering, so $\Aut(X, <)$ is trivial, hence extremely amenable, and $\Sym(X)$ acts uniquely transitively on the space of linear orders, hence $\LO(X) \simeq \Sym(X)$.

	When $X$ is infinite, any linear order which is \emph{$\omega$-homogeneous} works, that is, any linear order such that $\Aut(X, <)$ acts transitively on the set $\Power_n (X)$ of subsets of $X$ of size $n$, for any natural number $n$. The extreme amenability of these groups has been established in \cite[5.4]{Pestov_1998} (see also Assertion~5.1 therein for the existence of these $\omega$-homogeneous orders). The fact that they can be used to describe the universal minimal flow of $\Sym(X)$ can be read in \cite[\S~6.3]{Pestov_2006} (see in particular 6.3.5 and 6.3.6).
\end{proof}

Propositions~\ref{prop:symamenrprec} and~\ref{prop:umfsym} hold (trivially) for finite $X$. A specific property of infinite symmetric groups is that they are \emph{topologically simple}, that is, any nontrivial normal subgroup is dense. This follows from the classification of normal subgroups of $\Sym(X)$ \cite[(4), p.~17]{Baer_1934} (see also \cite[\S~8.1]{DM_1996} for a more recent exposition). In particular, since they are not commutative, they are \emph{topologically perfect} (their derived subgroup\footnote{Not to be confused with the derived space of a topological space, the \emph{derived subgroup} of a group is the subgroup generated by all commutators.} is dense).

On the contrary, for finite $X$, the derived subgroup of $\Sym(n)$ is $\Alt(n)$ (the alternating subgroup, made of even permutations), which is nontrivial for $n \geq 3$. It is also the only proper nontrivial normal subgroup, except for $n = 4$ where a second one appears, generated by the involutions without fixed points.

\section{Dense subgroups and upper completion}\label{sec:densecompl}

Let $G$ be a topological group and $H$ a \emph{dense} subgroup of $G$. We endow the group $H$ with the induced topology of $G$. For many dynamical considerations, there is no difference between $G$ and $H$. More precisely:

\begin{proposition}\label{prop:dense}
Let $G$ be a topological group and $H$ a dense subgroup of $G$, endowed with the induced topology. Then
	\begin{enumerate}
		\item $G$ is topologically perfect if and only if $H$ is so;
		\item\label{it:amendense} $G$ is amenable if and only if $H$ is so;
		\item $G$ is Roelcke-precompact if and only if $H$ is so;
		\item The universal minimal flow of $G$ is the same as the one of $H$ (that is, $\umf{H}$ is the space $\umf{G}$ endowed with the $H$-action deduced from the $G$-action).
	\end{enumerate}
\end{proposition}

\begin{proof}
\begin{enumerate}
	\item  Being topologically perfect for the induced topology means that $H \cap \overline{[H, H]} = H$ (where the closure is taken in $G$), i.e. $\overline{[H, H]} \supseteq H$. But since $H$ is a dense subgroup of $G$, we have
		\begin{equation*}
			\overline{[G, G]} \supseteq \overline{[H, H]} \supseteq [\overline{H}, \overline{H}] = [G, G],
		\end{equation*}
		i.e. $\overline{[G, G]} = \overline{[H, H]}$, hence topological perfectness of $H$ is indeed equivalent to the one of $G$.
	\item Let $G$ act continuously on a nonempty convex compact set $K$. If $H$ is amenable, then it has a fixed point $x \in K$. But since the action is continuous and $H$ is dense, the point $x$ is also fixed by $G$. Hence $G$ is amenable.

Conversely, assume $G$ is amenable and let $H$ act continuously on a nonempty convex compact set $K$. Since the group $\Homeo(K)$ is \Raikov-complete (\cite[Th.~6]{Arens_1946}, see also \cite[\S~2.3]{Gheysens_ulb}), we can extend the morphism $\phi\colon H \rightarrow \Homeo(K)$ to the upper completion $\ucompl{H}$, which contains $G$ (since $\ucompl{H}$ is actually isomorphic to the upper completion $\ucompl{G}$ of $G$, thanks to the universal property defining completions). In particular, we can extend the $H$-action on $K$ to a $G$-action. The latter has a fixed point by assumption, which is a fortiori fixed by $H$. 
	\item The Roelcke uniform structure of $H$ is equal to the restriction on $H$ of the Roelcke uniform structure of $G$ \cite[3.24]{RD_1981}. Therefore, again thanks to the universal property defining completions, the Roelcke completion $\Rcompl{H}$ of $H$ is isomorphic to the Roelcke completion $\Rcompl{G}$ of $G$. Hence $G$ is Roelcke-precompact if and only if $H$ is so.
	\item As for \eqref{it:amendense}, this fact follows from the possibility to extend any continuous morphism $\phi\colon H \rightarrow \Homeo(K)$ (for $K$ compact) to the upper completion of $H$, hence in particular to $G$. This brings a canonical bijection between $H$-flows and $G$-flows. Since $H$ is dense in $G$, a flow is minimal as an $H$-flow if and only if it is so as a $G$-flow. In particular, $\umf{H}$ is the space $\umf{G}$ endowed with the $H$-action.
\end{enumerate}
\end{proof}

Another way to phrase the above Proposition is to say that the notions considered therein can be studied directly on the upper completion of a topological group $G$ and that if $H$ is a dense subgroup of $G$, then $H$ and $G$ have isomorphic upper completions.

\begin{remark}
	Of course, this only works if $H$ is endowed with the \emph{induced} topology coming from $G$. For instance, a compact group (hence an amenable, Roelcke-precompact group which is equal to its universal minimal flow) can well contain a dense nonabelian free subgroup $F$ (take for instance the profinite completion of $F$). But with the discrete topology, $F$ is a non-amenable, non-Roelcke-precompact group with a nonmetrisable universal minimal flow. (Compact groups containing dense free subgroups are actually quite common, the interested reader will find various linear examples in~e.g.~\cite{Harpe_1983} or \cite{BG_2007}.)
\end{remark}

\section{Fully transitive groups of homeomorphisms}\label{sec:fullytransitive}

Let $X$ be any topological space. The (abstract) group $\Homeo(X)$ is naturally a subgroup of the group $\Sym(X)$ of all bijections in $X$. When $X$ is scattered, this inclusion is actually a topological embedding (when both groups are endowed with their respective pointwise convergence topology). As we saw in Proposition~\ref{prop:dense}, the notions we are interested in can be studied equivalently on $\Homeo(X)$ or on its closure in $\Sym(X)$. In order to describe the latter, we must understand what kind of permutations can be approximated (in the pointwise topology) by homeomorphisms, that is, we must understand what are the orbits of tuples of points of $X$ under the action of $\Homeo(X)$. To that end, we introduce in this section the notion of full transitivity.

\subsection{Similarity and full transitivity}

The notions of this subsection make sense for the abstract group $\Homeo(X)$ of any topological space---but will prove useful in the next subsections only when $X$ is scattered.

\begin{definition}
	Let $X$ be a topological space. We say that two points $x$ and $y$ in $X$ are \emph{similar} if there exist respective neighbourhoods $U_x$ and $U_y$ and a homeomorphism $h$ from $U_x$ onto $U_y$ such that $h(x) = y$. The equivalence classes for this relation are called \emph{similarity classes}.
\end{definition}

Observe that the homeomorphism $h$ is not required to be defined on the whole space. For instance, if $X$ is a simplicial regular graph, then all vertices are similar, even if the graph is not transitive. On the other hand, of course, homeomophisms have to respect the similarity relation, hence any notion of transitivity for homeomorphism groups has to take this relation into account.

\begin{remark}\label{rem:similarityordinal}
    Obviously, two similar points have the same Cantor--Bendixson rank. On the other hand, two points of rank zero (that is, two isolated points) are similar but this does not hold for higher rank, as can easily be seen by considering for $Y$ the one-point compactification of a countable discrete space, for $Z$ the one-point compactification of an uncountable discrete space and for $X$ the disjoint union of $Y$ and $Z$. Then $X$ has (exactly) two points of rank one but they are not similar since one has a basis of countable neighbourhoods whereas the other one only has uncountable ones.

	However, to make a bridge with the companion paper \cite{Gheysens_omega}, let us observe that similarity does boil down to having the same Cantor--Bendixson rank for \emph{ordinal spaces}. Indeed, as was observed there, a point $x$ in an odinal space has rank $\alpha > 0$ if and only if it can be written as $x' + \omega^\alpha$. Hence, it has a neighbourhood homeomorphic to $[0, \omega^\alpha]$, with $x$ identified with $\omega^\alpha$. Therefore, all points of rank $\alpha$ are similar to each other. The simplicity of the similarity relation in this context is fruitful, as we shall see in Section~\ref{sec:classif}.
\end{remark}

Let us now define the main notion of this paper.

\begin{definition}
	Let $X$ be a topological space and $G$ be an abstract group acting on $X$ by homeomorphisms. We say that $G$ is \emph{fully transitive} if it is as transitive as allowed by the similarity relation, that is: for any natural number $k$, for any $k$-tuples $(x_1, \dots, x_k)$ and $(y_1, \dots, y_k)$ of distinct points such that $x_i$ and $y_i$ are similar, there exists $g \in G$ such that $g(x_i) = y_i$ for each $i$. By abuse, we will also say that a space is \emph{fully transitive} if its homeomorphism group $\Homeo(X)$ is so (which is of course equivalent to require the existence of any fully transitive group acting on $X$).
\end{definition}

The point of this definition is its immediate translation into a denseness result.

\begin{proposition}\label{prop:fulltransdense}
	Let $X$ be a topological space and $G$ be an abstract group acting faithfully on $X$ by homeomorphisms. Consider $G$ as a subgroup of $\Sym(X)$ and let $\overline{G}$ be its closure (for the topology of discrete pointwise convergence on $\Sym(X)$). Then $G$ is fully transitive if and only if $\overline{G}$ is equal to
	\begin{equation*}
		\prod_i \Sym(X_i),
	\end{equation*}
	where $X = \bigsqcup_i X_i$ is the partition of $X$ into similarity classes.\qed
\end{proposition}

An easy example of a space which fails to be fully transitivite is given by the real line (since a homeomorphism fixing $a < b$ cannot send a point of $[a, b]$ outside this interval, whereas all points are similar). On the other hand, we will see below (Section~\ref{sec:examples}) that full transitivity is quite common in scattered spaces and holds in particular for all locally compact scattered spaces (and therefore for all ordinal spaces). For now, we proceed with the impact of full transitivity on the structure of homeomorphism groups.

\subsection{Topological properties}

Let us explore the notions of the preceding subsection for the case where $X$ is scattered, so that $\Homeo(X)$ becomes a topological group. Note that, since the group $\Sym(X)$ is always \Raikov-complete (Proposition~\ref{prop:symamenrprec}), the closure of $\Homeo(X)$ in $\Sym(X)$, which is identified for fully transitive groups by Proposition~\ref{prop:fulltransdense}, is nothing but its upper completion.

For the sake of concision, we will state our results directly for the whole group $\Homeo(X)$. They would hold also for any fully transitive subgroup $H$ but this does not genuinely extend their scope, since $H$ would then a fortiori be dense in $\Homeo(X)$---hence would fit in the framework of Proposition~\ref{prop:dense}.

\begin{theorem}\label{th:amenRprec}
	Let $X$ be a fully transitive scattered space. Then the group  $\Homeo(X)$ is amenable and Roelcke-precompact. 
\end{theorem}

\begin{proof}
	By Propositions~\ref{prop:dense} and \ref{prop:fulltransdense}, we only need to prove these properties for arbitrary products of symmetric groups.  But both properties are preserved by products (see e.g.~\cite[Lemma~15]{Gheysens_omega} for amenability and \cite[3.35]{RD_1981} for Roelcke-precompactness) and we already observed in Proposition~\ref{prop:symamenrprec} that symmetric groups are amenable and Roelcke-precompact.
\end{proof}

The same strategy allows to compute the universal minimal flow:

\begin{theorem}\label{th:umflo}
	Let $X$ be a fully transitive scattered space. Then the universal minimal flow of $\Homeo(X)$ is 
\begin{equation*}
    \prod_i \LO(X_i),
\end{equation*}
	where $X = \bigsqcup_i X_i$ is the partition of $X$ into similarity classes.
\end{theorem}
(We recall that $\LO(X_i)$ is the space of linear orders on the set $X_i$, defined in Subsection~\ref{sec:sym}.)

\begin{proof}
	By Propositions~\ref{prop:dense} and~\ref{prop:fulltransdense}, we only need to prove that the above compact set is indeed the universal minimal flow of $\prod_i \Sym(X_i)$. This follows from Propositions~\ref{prop:umfproduct} and~\ref{prop:umfsym}.
\end{proof}

\begin{remark}
	In particular, for a fully transitive scattered space $X$, the universal minimal flow of $\Homeo(X)$ is metrisable if and only if all similarity classes of $X$ are countable and at most countably many of them are not singletons, e.g.~if $X$ is countable. 
	Note that a \emph{metrisable} countable scattered space is homeomorphic to a subspace of an ordinal space \cite[IV]{KU_1953} (the \enquote{countable} assumption is actually not necessary here, see \cite[Cor.~5]{Telgarsky_1968}). There exist however countable scattered space that cannot be embedded into an ordinal space. For instance, there exist points $p \in \SC \Nat$ (the Stone--\v{C}ech compactification of the discrete space $\Nat$) such that the (countable and scattered) subspace $\Nat \cup \{p\}$ does not admit any scattered compactification, hence cannot be embedded into an ordinal space, cf.~\cite{Solomon_1976, Telgarsky_1977b}. 
\end{remark}

\subsection{Topological structure}

We now investigate the closed normal subgroups of a fully transitive group.

\begin{proposition}
	Let $X$ be a fully transitive scattered space. Assume that each similarity class is infinite. Then $\Homeo(X)$ is topologically perfect.
\end{proposition}

\begin{proof}
	By Proposition~\ref{prop:dense}, we only need to prove that an arbitrary product of infinite symmetric groups is topologically perfect. But $\Sym(X)$ is so whenever $X$ is infinite and a product of topologically perfect groups is again topologically perfect.
\end{proof}

We can actually classify all closed normal subgroups. For a subset $A$ of $X$, we write $\Fixa(A)$ for its fixator, i.e.~the subgroup of $\Homeo(X)$ fixing pointwise $A$ (if $A$ is empty, $\Fixa(A)$ is then the whole $\Homeo(X)$).

\begin{proposition}\label{prop:classnormal}
	Let $X$ be a fully transitive scattered space. Assume that each similarity class is infinite. Then any closed normal subgroup of $\Homeo(X)$ is of the form $\Fixa(A)$, where $A$ is an arbitrary union of similarity classes.
\end{proposition}

\begin{proof}
	It is obvious that such fixators are closed and normal. Conversely, let $N$ be a closed normal subgroup of $\Homeo(X)$ and let $A$ be the set of points fixed by $N$. In particular, $N$ is a closed subgroup of $\Fixa(A)$. It is sufficient to show that $N$ is dense in $\Fixa(A)$ and that $A$ is a union of similarity classes.

	By normality, $A$ has to be invariant by $\Homeo(X)$: full transitivity then forces $A$ to be saturated for the similarity relation, i.e.~to be a union of similarity classes. Let us write $A = \bigcup_J X_i$, where $X = \bigsqcup_I X_i$ is the decomposition of $X$ into similarity classes and $J \subseteq I$ are sets of indices. By definition of $A$, the group $N$ is (in $\Sym(X)$) a subgroup of 
	\begin{equation*}
		H = \prod_{i \in J} \{\id_{X_i}\} \times \prod_{i \in I \setminus J} \Sym(X_i)
	\end{equation*}
	and, for any $i \not\in J$, the image of $N$ in $\Sym(X_i)$ is not trivial.

	But the closure $\overline{N}$ of $N$ in the group $\prod_I \Sym(X_i)$ is a closed normal subgroup (as $\Homeo(X)$ is dense in the latter). Since $\Sym(X_i)$ is topologically simple, we have $\overline{N} = H$. By construction, we have $N \leq \Fixa(A) \leq H$, hence $N$ is dense in $\Fixa(A)$, as expected. 
\end{proof}

\begin{remark}\label{rem:classnormal}
	For ease of statement, we have assumed here that each similarity class is infinite, so that the corresponding symmetric groups are topologically simple. Knowing that the derived subgroup of $\Sym(n)$ is $\Alt(n)$ (which is nontrivial for $n \geq 3$), it is easy to describe the closed derived subgroup of $\Homeo(X)$ in full generality: it is made of all homeomorphisms that induce an even permutation on the finite similarity classes.

	Similarly, knowing the proper nontrivial normal subgroups of $\Sym(n)$, we can also describe the closed normal subgroups of $\Homeo(X)$ in full generality. They are made of all homeomorphisms that are trivial on $\bigcup_J X_i$, induce even permutations on $X_i$ for $i \in K$ and induce the identity or a fixed-point free involution on $X_i$ for $i \in L$, where $J$, $K$ and $L$ are three disjoint subsets of the index set of all similarity classes, such that $X_i$ is finite of size at least~$3$ for $i \in K$ and finite of size $4$ for $i \in L$ ($J$, $K$ or $L$ may be empty).
\end{remark}

\subsection{Remarks on the abstract structure}\label{sec:abstract}

For the sake of completeness, let us observe that the results of this section (or indeed of this whole paper) only reveal information on the group $\Homeo(X)$ as a topological group and not as an abstract group.

As a first illustration, let us observe that the group $\Homeo(X)$ is never abstractly amenable when $X$ is infinite, scattered, Hausdorff, and completely regular\footnote{\emph{Complete regularity} means that points can be separated from closed sets by continuous real-valued functions or, equivalently, that the topology can be defined by a uniform structure. It implies that \emph{closed} neighbourhoods of a point form a basis of neighbourhoods.}. Indeed, either $X$ is just an infinite discrete space (in which case $\Homeo(X) = \Sym(X)$ contains nonabelian free subgroups) or $X$ contains a point $x$ of Cantor--Bendixson rank one. Since $x$ is of rank one, there must exist a subset $V$ of isolated points such that $V \cup \{x\}$ is a neighbourhood of $x$. Furthermore, $V$ must be infinite since $X$ is Hausdorff and $x$ is not isolated. By complete regularity, we may further assume that $V \cup \{x\}$ is closed. Therefore, any permutation of $V$ can be extended to a homeomorphism of $X$ (fixing the complement of $V$). So $\Homeo(X)$ contains a copy of $\Sym(V)$, which itself contains nonabelian free subgroups.

As a second illustration, observe that the subgroup of finitely-supported homeomorphisms of an infinite scattered space $X$ is a normal subgroup which is not closed (its closure is the subgroup $\Fixa(X^{(1)})$). More interestingly, let $\Lambda_\lambda$ denote the increasing union of $\Fixa(X^{(\alpha)})$ for $\alpha < \lambda$, where $\lambda$ is a limit ordinal. Then $\Lambda_\lambda$ is a dense normal subgroup of $\Fixa(X^{(\lambda)})$, in general not equal to it. It has been shown for instance that \emph{any} countable group can appear as an abstract quotient of the form $\Homeo(X) / \Lambda_{\omega_1}$, for $X$ a compact scattered space of Cantor--Bendixson rank $\omega_1 + 1$ \cite[Th.~3.9]{DS_1992}.

\section{Examples and counterexamples of fully transitive groups}\label{sec:examples}

We now finally show that full transitivity occurs under mild topological assumptions on the scattered space $X$, for instance local compactness. On the other hand, Subsection~\ref{sec:exotic} is devoted to the constructions of very general scattered spaces that show that amenability and Roelcke-precompactness of $\Homeo(X)$ do not hold for \emph{all} scattered spaces.

\subsection{Zero-dimensional spaces}

Here is the key technical result of this paper.

\begin{theorem}\label{th:zerodim}
	Let $X$ be a Hausdorff zero-dimensional space\footnotemark. Then $\Homeo(X)$ is fully transitive.
\end{theorem}
\footnotetext{Recall that a $\Ti[0]$ zero-dimensional space is automatically Hausdorff.}

\begin{proof}
  By induction, it is enough to show that for any finite set $F$ of $X$, the fixator $\Fixa(F)$ is transitive on $X_i \setminus F$ for each similarity class $X_i$.

Let thus $x$ and $y$ be two similar points not in $F$. By definition, there exist two respective neighbourhoods $U_x$ and $U_y$ and a homeomorphism $h\colon U_x \rightarrow U_y$ such that $h(x) = y$. Since $X$ is Hausdorff, we can assume that $U_x$ and $U_y$ are disjoint and avoid the finite set $F$. Moreover, since $X$ is zero-dimensional, we can also assume that $U_x$ and $U_y$ are clopen. Therefore, the map $g$ defined by
\begin{equation*}
    g(z) = \begin{cases}
h(z) &\text{if } z \in U_x \\
h^{-1}(z) &\text{if } z \in U_y \\
z &\text{otherwise} \\
\end{cases}
\end{equation*}
is a homeomorphism of $X$ that fixes $F$ and swaps $x$ and $y$.
\end{proof}

\begin{remark}
	In a scattered space $X$, if the Cantor--Bendixson rank of a point $x$ is $\alpha$, then there must exist a neighbourhood $U$ of $x$ such that all points of $U \setminus \{x\}$ have rank $< \alpha$. In particular, in the above proof, we can make it so that the homeomorphism $g$ is trivial on points of sufficiently high rank. In particular, if $X^{(\lambda)}$ is a singleton for some limit ordinal $\lambda$, then the subgroup $\Lambda_\lambda$ of $\Homeo(X)$ is already fully transitive, where $\Lambda_\lambda$ is the union of $\Fixa(X^{(\alpha)})$ for $\alpha < \lambda$ (compare with Subsection~\ref{sec:abstract}). 
\end{remark}

\begin{remark}
	A fitting version of full transitivity was established and used for ordinal spaces in \cite[Lemma~12]{Gheysens_omega}. Indeed, when $X$ is an ordinal space, two points are similar if and only if they have the same Cantor--Bendixson rank (Remark~\ref{rem:similarityordinal}).
\end{remark}

In particular, all the results of Section~\ref{sec:fullytransitive} apply to zero-dimensional scattered spaces. Let us point out:

\begin{corollary}\label{cor:loccompscatt}
	Let $X$ be a scattered space which is locally compact \emph{or} metrisable. Then $\Homeo(X)$ is amenable and Roelcke-precompact. Moreover, its universal minimal flow is
\begin{equation*}
    \prod_i \LO(X_i),
\end{equation*}
	where $X = \bigsqcup_i X_i$ is the partition of $X$ into similarity classes.
\end{corollary}

\begin{proof}
    By Theorems~\ref{th:zerodim}, \ref{th:amenRprec} and \ref{th:umflo}, we only need to show that $X$ is zero-dimensional.
	\begin{itemize}
		\item A locally compact scattered space is zero-dimensional by van Dantzig's theorem, since a scattered $\Ti[1]$ space is totally disconnected.
		\item Since zero-dimensionality is inherited by subspaces, any scattered space that admits a \emph{scattered} compactification has to be zero-dimensional by the previous argument. By \cite[Cor.~5]{Telgarsky_1968}, metrisable scattered spaces admit such compactifications.
	\end{itemize}
\end{proof}

We note that locally compact scattered spaces include the following two much-studied classes:
\begin{enumerate}
	\item Ordinal spaces and, more generally, any locally closed subspace of a finite product of ordinals. We retrieve in particular Theorems~13 and 16 about $\Homeo(\omega_1)$ of \cite{Gheysens_omega}. The computation of the universal minimal flow of $\Homeo(\omega_1)$ was also obtained independently by Basso and Zucker in \cite[\S~9]{BZ_2020}.
	\item Compact scattered spaces. By Stone's duality, these correspond to \emph{superatomic} Boolean algebras (Boolean algebras whose any quotient is atomic). This class of algebras has been extensively investigated, in particular in connection with the axioms of set theory---see for instance the survey \cite{Roitman_1989}.
\end{enumerate}

\subsection{Exotic constructions}\label{sec:exotic}

We prove here that the previous results cannot be extended to a larger class of scattered spaces in the absence of full transitivity. More precisely, Corollary~\ref{cor:loccompscatt} can fail for the homeomorphism group of a scattered space that is Hausdorff but not locally compact or that is locally quasi-compact but not Hausdorff. To that end, we will give two constructions that show the variety of groups that can occur as homeomorphism groups of scattered non-locally compact spaces:

At the core of our constructions lies the notion of a \emph{graph}, which for this paper is assumed to be non-oriented, simple and combinatorial, i.e.~a graph $\Graph$ is a pair $(V, E)$ of a vertex set $V$ and an edge set $E$ that is a subset of the set of all pairs of (distinct) points of $V$. We recall that the automorphism group $\Aut(\Graph)$ is the subgroup of $\Sym(V)$ that leaves $E$ invariant\footnote{To avoid trivialities, we will implicitly assume that the edge set is not empty. This is actually not a restriction, since the automorphism group of the graph without edge over the vertex set $V$ is the same as the automorphism group of the graph on $V$ where all pairs of vertices are in the edge set.}. Its topology is thus the topology of pointwise convergence on the set $V$ (viewed as a discrete space). Whenever the graph $\Graph$ is locally finite (each vertex has only a finite number of neighbourhoods), the group $\Aut(\Graph)$ is locally compact. Moreover, if $\Graph$ is a Cayley graph of a finitely generated group $G$, then the action of $G$ on $\Graph$ by translations realizes $G$ as a discrete subgroup of $\Aut(\Graph)$. In particular, the class of automorphism groups of graphs contains many non-amenable, non-Roelcke-precompact groups whose universal minimal flow is not metrisable, for instance those of Cayley graphs of non-amenable finitely generated groups (e.g.~$\Aut(\Tree_{2d})$ for any regular tree $\Tree_{2d}$ of degree $2d$, $d \geq 2$).

Both our constructions will show that the automorphism group of \emph{any} graph can occur as the homeomorphism group of some scattered spaces, when we leave the realm of locally compact or metrisable scattered spaces. By the previous paragraph, this will be sufficient to show the importance of the hypothesis in Corollary~\ref{cor:loccompscatt}.

We first provide examples that fail to be Hausdorff.

\begin{proposition}\label{prop:exoticgraph}
	Any automorphism group of a graph can be realized as the homeomorphism group of a scattered, locally quasi-compact space (which is not $\Ti[1]$ and not zero-dimensional).
\end{proposition}

\begin{proof}
    Let $\Graph$ be a graph with vertex set $V$ and edge set $E$. 
The idea of the construction is to realize $V$ as the dense set of isolated points of a scattered space $X$ in such a way that the adjacency relation given by $E$ is topologically encoded in $X$. To do that, we define $X$ as the set $V \cup E$ endowed with the following topology: a subset $U$ is open if, whenever $e \in U$ for some $e \in E$, then $e \subset U$. That is, whenever $U$ contains a point corresponding to an edge, it contains also the points associated to the vertices of that edge.

Observe that, with this topology, the set of isolated points of $X$ is exactly $V$. Moreover, $X^{(1)} = E$ and $X^{(2)} = \emptyset$, hence $X$ is indeed scattered. Since any point admits a finite neighbourhood ($\{v\}$ if $v \in V$ or $\{e\} \cup e$ if $e \in E$), the space $X$ is locally quasi-compact.

As for any scattered space, the restriction map
\begin{equation*}
    \rho\colon \Homeo(X) \rightarrow \Sym(X \setminus X^{(1)}) = \Sym(V)
\end{equation*}
	is a continuous and injective morphism, since isolated points form a discrete dense subset. Let us show that $\rho$ is, here, an isomorphism onto its image, which is $\Aut(\Graph)$.

	Whereas points in $X^{(1)}$ are closed, the closure of the singleton $\{v\}$ for $v \in V$ is made of the vertex $v$ and of all the edges $e$ that contain $v$. (In particular, $X$ is not $\Ti[1]$.) That is, the vertices $v$ and $w$ are adjacent in the graph $\Graph$ if \emph{and only if} the intersection $\overline{\{v\}} \cap \overline{\{w\}}$ (in the topological space $X$) is not empty. In particular, for any homeomorphism $h$ of $X$, the restriction $\rho(h)$ preserves the adjacency relation $E$ on $V$. On the other hand, any element of $\Aut(\Graph)$, when viewed as a permutation on $V \cup E$, preserves the open sets. Therefore, $\rho$ is a bijection onto the subgroup $\Aut(\Graph)$ of $\Sym(V)$.

	Finally, observe that a homeomorphism of $X$ fixes an edge $e$ if it fixes both its vertices. Therefore, a basis of identity neighbourhoods of $\Homeo(X)$ is here given by the fixators of finite sets of isolated points, hence the map $\rho$ is a topological isomorphism onto its image, i.e.~$\Homeo(X) \simeq \Aut(\Graph)$.
\end{proof}

We now provide examples that are Hausdorff but fail to be zero-dimensional (and a fortiori fail to be locally compact). For this second construction, we recall that we can realize a graph as a topological space by identifying each edge with an interval $[0,1]$ and glueing the extremities of the intervals along the vertices. Notice that a homeomorphism of the resulting topological space has to preserve vertices of degree other than $2$. In particular, if the graph $\Graph$ has no such vertices, there exists a surjection from the homeomorphism group $\Homeo(\Graph)$ onto the automorphism group $\Aut(\Graph)$. 

\begin{proposition}\label{prop:exoticspace}
	Let $Z$ be any Hausdorff completely regular (non-discrete) topological space. Then there exists a Hausdorff (non-locally compact, non-zero-dimensional) scattered space $X$ such that $\Homeo(Z)$ is abstractly isomorphic to the quotient $\Homeo(X) / \Fixa(X^{(1)})$.

 Moreover, if $Z$ is the topological realization of a graph $\Graph$ without vertices of degree~$2$, then the composition of this isomorphism with the map $\Homeo(Z) \rightarrow \Aut(\Graph)$ (given by the restriction of a homeomorphism to its action on the vertices) gives a continuous surjection from $\Homeo(X)$ onto $\Aut(\Graph)$.
\end{proposition}

\begin{proof}
    Let $X$ be the set $Z \times [0, \omega]$ and let us write $Z_\omega$ for the slice $Z \times \{\omega\}$. (Here, $[0, \omega]$ is the countable set $\{0, 1, 2, \dots\} \cup \{\omega\}$, i.e.~the ordinal $\omega + 1$.) We endow $X$ with the following topology: a set $U$ is open if, whenever it contains a point $(z, \omega)$, then it contains $V \times [n, \omega[$ for some $n < \omega$ and some $V \subseteq Z$ which is a neighbourhood of $z$ (for the given topology on $Z$).

	Observe that the set of isolated points of $X$ is $Z \times [0, \omega[$ and that any point $(z, \omega)$ admits a neighbourhood $U$ such that $U \cap Z_{\omega} = \{(z, \omega)\}$ (namely, $U = \{(z, \omega)\} \cup V \times [n, \omega[$ for any $n < \omega$ and any $V$ which was a neighbourhood of $z$ in $Z$). Hence $X^{(1)} = Z_\omega$ and $X^{(2)} = \emptyset$, which shows that $X$ is scattered. It is obvious from the definition of the open sets that $X$ is Hausdorff (since we assumed $Z$ to be so).

The key feature of this construction is that, even if $Z_{\omega}$ is a discrete subspace of $X$, its embedding in $X$ still carries enough topological information to recover the given topology on $Z$. More precisely, let $V$ be an open subset of $Z$ (for its given topology) and consider, in $X$, the set $W = V \times [n, \omega[$. Then $W \cup (z, \omega)$ is a neighbourhood of $(z, \omega)$ in $X$ for \emph{any} $z \in V$. Better, the intersection of $Z_\omega$ with the closure of $W$ (for the topology on $X$) is exactly the set $\overline{V} \times \{\omega\}$, where $\overline{V}$ is the closure of $V \subseteq Z$ for the given topology on $Z$. This has two consequences:
\begin{enumerate}
    \item If $z \in Z$ is not isolated, then $(z, \omega) \in X$ admits no basis of \emph{closed} neighbourhoods. Indeed, a closed neighbourhood $U$ of $(z, w)$ must meet $Z_\omega$ in a set in bijection with a closed neighbourhood of $z \in Z$, hence in particular $U \cap Z_\omega$ has at least two points since $z$ is not isolated. But the neighbourhood $W \cup (z, w)$ of above meets $Z_\omega$ in only one point. A fortiori, since $Z$ is not discrete, $X$ cannot be locally compact nor zero-dimensional (nor even completely regular).
    \item Under the obvious identification between the sets $X^{(1)} = Z_\omega$ and $Z$, the restriction map $\Homeo(X) \rightarrow \Sym(Z)\colon h \mapsto h_{|X^{(1)}}$ actually defines a map 
\begin{equation*}
    \rho\colon \Homeo(X) \rightarrow \Homeo(Z).
\end{equation*}
Indeed, since $Z$ is completely regular, any permutation that preserves the closed neighbourhoods must be a homeomorphism. Moreover, this map is surjective (for any $h \in \Homeo(Z)$, the map $\widehat{h} = h \times \id_{[0, \omega]}$ is a homeormorphism of $X$ such that $\rho(\widehat{h}) = h$).
\end{enumerate}

The kernel of $\rho$ is made of all homeomorphisms that are trivial on $X^{(1)}$. This gives us the abstract isomorphism of the proposition. The continuity statement in case where $Z$ is a simplicial graph is obvious from the fact that the topology on the automorphism group of $\Graph$ is given by the pointwise convergence on the discrete set of vertices.
\end{proof}

\begin{remark}
	A contrario, the above construction can also yield full transitivity outside the realm of zero-dimensional spaces. For instance, if $Z$ is a perfect space on which $\Homeo(Z)$ acts $n$-transitively for any $n$ (e.g.~$Z$ can be a Euclidean space of dimension at least two), then $\Homeo(X)$ is fully transitive whereas $X$ is not zero-dimensional. 
\end{remark}

\section{Application to the classification of homeomorphism groups of ordinal spaces}\label{sec:classif}

As an application of the topological structure of fully transitive groups, we classify here the homeomorphism group of compact ordinal spaces.

We call \emph{ordinal space} any ordinal $\alpha$ endowed with its order topology. The latter is generated by open intervals; since $\alpha$ is well-ordered, this topology is Hausdorff and admits a basis made of the sets
\begin{equation*}
	[\beta, \gamma[ = \left\{x\ \middle|\ \beta \leq x < \gamma \right\}
\end{equation*}
where $\beta < \gamma \leq \alpha$. Closed intervals $[\beta, \gamma]$ are easily seen to be compact; since $[\beta, \gamma] = [\beta, \gamma + 1[$, we conclude that ordinal spaces are locally compact. Any nonempty subset $A$ of $\alpha$ contains a minimum, which is of course isolated in $A$. In particular, ordinal spaces are Hausdorff, zero-dimensional, and scattered; their homeomorphism group is fully transitive.

By definition, an ordinal is the set of all ordinals strictly smaller than itself. Hence a more suggestive writing for the topological structure of ordinals is $\alpha = [0, \alpha[$ and $\alpha + 1 = [0, \alpha]$. In particular, an ordinal is compact if and only if it is $0$ or a successor ordinal. Moreover, any compact subset of a limit ordinal $\alpha$ is contained in a compact set of the form $[0, \beta]$. Therefore, Alexandrov's one-point compactification of $\alpha$ can be identified with $\alpha + 1$.

The main topological properties of ordinals spaces can be found in Examples~40--43 of \cite{SS_1978}. For basic facts about ordinal arithmetic, we refer to \cite[I.2]{Jech_2003}, from which we also follow the notational convention: in particular, the addition is \enquote{continuous on the right} ($\beta \mapsto \alpha + \beta$ is continuous for any fixed $\alpha$; equivalently, it commutes with the supremum).


We point out three facts particularly relevant for the classification of homeomorphism groups. Firstly, any nonzero ordinal can be uniquely written as a sum of decreasing powers of $\omega$, called its \emph{Cantor normal form}. That is, for any $\alpha > 0$, there exist a unique natural number $n \geq 1$, a unique decreasing $n$-tuple of ordinals $\alpha \geq \beta_1 > \beta_2 > \dots > \beta_n$, and a unique $n$-tuple of nonzero natural numbers $k_1, \dots, k_n$ such that
\begin{equation*}
	\alpha = \omega^{\beta_1} \cdot k_1 + \dots + \omega^{\beta_n} \cdot k_n.
\end{equation*}

Secondly, powers of $\omega$ \enquote{absorb} smaller ordinals on the left, that is $\beta + \omega^\alpha = \omega^\alpha$ for any $\beta < \omega^\alpha$ (this is proven easily by induction on $\alpha$). Note that the decomposition into Cantor normal forms implies by the way that these powers are the only ordinals with this property.

Lastly, a nonzero point $x$ of an ordinal space is of Cantor--Bendixson rank $\alpha$ if and only if $x = x' + \omega^\alpha$ for some ordinal $x'$. The \enquote{if} is an easy transfinite induction argument and the \enquote{only if} follows from the Cantor normal form.

\subsection{Classification of ordinal spaces}

Let us start by classifying the ordinal spaces up to homeomorphism. The following result is probably well-known (the compact case is given for instance in \cite[Cor.~3]{Baker_1972}).

\begin{theorem}\label{th:classordspaces}
    Any ordinal space is homeomorphic to one of the following ordinal spaces:
\begin{enumerate}
    \item $\emptyset$ or $k$,
    \item $\omega^\alpha \cdot k + 1$,
    \item $\omega^\alpha \cdot k$,
    \item $\omega^\alpha \cdot k + \omega^\beta$,
\end{enumerate}
where $k$ is a nonzero natural number, $\alpha$ is a nonzero ordinal and, for the last case, $\beta$ is nonzero ordinal $< \alpha$.

Moreover, there is no redundancy in this classification: no space of one of these four families is homeomorphic to a space of another family and, inside a family, the parameters $k$, $\alpha$, and $\beta$ are topological invariants.
\end{theorem}

\begin{proof}
If $\gamma$ is a finite ordinal, then it is obviously homeomorphic to a space of the first family. Let us then assume that $\gamma$ is infinite.

    The key point is to observe that ordinal addition translates to topological disjoint sum:
\begin{lemma}
    Let $\alpha$ and $\beta$ be two nonzero ordinals. Then the space $\alpha + \beta$ is homeomorphic to the disjoint sum of $\alpha + 1$ and $\beta'$, where $\beta' = \beta$ if $\beta$ is infinite and $\beta' + 1 = \beta$ if $\beta$ is finite.
\end{lemma}
\begin{proof}
    Since the subspace $[0, \alpha]$ of $\alpha + \beta$ is both open and closed, $\alpha + \beta$ is homeomorphic to the disjoint sum of $[0, \alpha]$ ($= \alpha + 1$) and $[\alpha + 1, \alpha + \beta[$. But the latter is homeomorphic to $[1, \beta[ = \beta'$.
\end{proof}
	Now of course the topological disjoint sum is \enquote{commutative}, hence the above lemma implies that, whenever $\alpha$ and $\beta$ are infinite, the space $\alpha + \beta + 1$ and $\beta + \alpha + 1$ are homeomorphic. Combining this with the \enquote{absorption phenomenon} for powers of $\omega$ and Cantor's normal form decomposition, we see that any infinite successor ordinal is homeomorphic to $\omega^\alpha \cdot k + 1$, where $\omega^\alpha \cdot k$ is the highest term in its Cantor normal form. 

We have thus shown that any compact (that is, successor) ordinal is homeomorphic to a space of one of the first two families. For the non-compact case, observe that a double application of the above lemma also allows to find a homeomorphism between $\alpha + \beta + \gamma$ and the topological disjoint sum of $\alpha + 1$, $\beta + 1$ and $\gamma'$, hence the spaces $\alpha + \beta + \gamma$ and $\beta + \alpha + \gamma$ are homeomorphic. We can therefore again use the absorption phenomenon on the Cantor normal form of a limit ordinal: if it is not already equal to $\omega^\alpha \cdot k$, then it is homeomorphic to $\omega^\alpha \cdot k + \omega^\beta$, where $\alpha$ (resp.~$\beta$) is the highest (resp.~lowest) exponent appearing in its Cantor normal form.

Let us prove that there is no redundancy in this classification. The various parameters $\alpha$, $\beta$, and $k$ are topological invariants. In the first family, $k$ is the cardinal of the space. Spaces of the second family have Cantor--Bendixson rank $\alpha + 1$ and they have exactly $k$ points in their $\alpha$-th derived subset. The same is true for the \emph{one-point compactification} of spaces of the third family. And spaces of the fourth family have exactly $k$ points of rank $\alpha$ and the point added in their one-point compactification is of rank $\beta$. Hence no two spaces of the same family can be homeomorphic to each other.

Lastly, no homeomorphism can occur across these families. The spaces of the first family are finite and the ones of the second family are infinite and compact, whereas the last two families are made of non-compact spaces. A hypothetical homeomorphism $\phi$ between a space $\omega^\alpha \cdot k$ and a space $\omega^{\alpha'} \cdot k' + \omega^\beta$ would need to extend to one-point compactifications, which have respective Cantor--Bendixson rank $\alpha + 1$ and $\alpha' + 1$, hence $\alpha = \alpha'$. But we must also have $\alpha = \beta$, since these are the respective Cantor--Bendixson rank of the added point in the compactification. This would contradict $\alpha' > \beta$, thus such a homeomorphism cannot exist.
\end{proof}

\subsection{Classification of homeomorphism groups of ordinal spaces}

We now leverage the results about full transitivity to show that non-homeomorphic compact ordinal spaces have non-isomorphic homeomorphism groups. Recall that ordinal spaces are locally compact, hence their homeomorphism group is fully transitive (Corollary~\ref{cor:loccompscatt}).

Let us first observe that, for an ordinal space $X$, the classification of closed normal subgroups (Proposition~\ref{prop:classnormal} and Remark~\ref{rem:classnormal}) takes a particularly easy-to-state form. As we have already noticed, two points in an ordinal space are similar if and only if they have the same Cantor--Bendixson rank (Remark~\ref{rem:similarityordinal}). Hence all similarity classes are infinite, except possibly one (the set of points of rank $\CB(X) - 1$, if $\CB(X)$ is a successor). Moreover, if a homeomorphism fixes the similarity class $X^{(\alpha)} \setminus X^{(\alpha + 1)}$, it actually fixes the whole derived subspace $X^{(\alpha)}$: indeed, in a scattered space, the points of rank $\alpha$ are dense in the $\alpha$-th derived subspace and since $X$ is Hausdorff, the set of fixed points of a homeomorphism is a closed subspace. 

Therefore, the closed normal subgroups of the homeomorphism group of an ordinal space $X$ are exactly the fixators $\Fixa X^{(\alpha)}$, plus possibly the preimage of the one or two proper nontrivial normal subgroups of $\Sym(X^{(\beta)})$, if $\beta + 1 = \CB(X)$ and if the set $X^{(\beta)}$ contains at least three points (Remark~\ref{rem:classnormal}). Observe in particular that the closed normal subgroups form a nested sequence indexed by an ordinal. This enables us to classify the homeomorphism group in the compact case:


\begin{theorem}
   Let $G_{\alpha, k} = \Homeo([0, \omega^\alpha \cdot k])$, where $\alpha$ is an ordinal and $k$ is a nonzero natural number. Then $\alpha$ and $k$ are invariants of topological groups (that is, if $G_{\alpha, k}$ and $G_{\alpha', k'}$ are isomorphic as topological groups, then $\alpha = \alpha'$ and $k = k'$). In particular, if two compact ordinals have isomorphic homeomorphism groups, then they are homeomorphic.
\end{theorem}

\begin{proof}
    If $X = [0, \omega^\alpha \cdot k]$, then $X^{(\alpha)} = \{\omega^\alpha \cdot 1, \dots, \omega^\alpha \cdot k\}$ and the Cantor--Bendixson rank of $X$ is $\alpha + 1$. Thanks to the classification of closed normal subgroups explained above, we can then retrieve $k$ and $\alpha$ as invariants of topological groups as follows:
\begin{itemize}
    \item $k!$ is the maximal cardinal of a finite discrete quotient of $G_{\alpha, k}$;
    \item For any topological group $G$, let us define $\chain(G)$ as the supremum of all ordinals $\beta$ such that there exists an injective increasing map from the ordered set $(\beta, <)$ to the set of \emph{infinite-index} closed normal subgroups of $G$ (ordered by inclusion). This is an invariant of topological isomorphism and, by the classification of closed normal subgroups, we see that $\chain(G_{\alpha, k}) = \alpha$.
\end{itemize}
	The last statement follows from the classification of ordinal spaces (Theorem~\ref{th:classordspaces}).
	\end{proof}

\begin{remark}
	For any ordinals $\beta < \alpha$, there are $\kappa$ points of Cantor--Bendixson rank $\beta$ in the ordinal space $[0, \omega^\alpha]$, where $\kappa$ is the cardinal of $\omega^\alpha$ (which is equal to the cardinal of $\alpha$ if $\alpha$ is infinite). In particular, for a fixed infinite cardinal $\kappa$, the groups $\Homeo([0, \omega^\alpha])$, where the ordinal $\alpha$ ranges from $\kappa$ to its cardinal successor $\kappa_+$, give us $\kappa_+$ non-isomorphic groups with isomorphic upper completions, namely, $\Sym(\kappa)^\kappa$ (and hence with homeomorphic universal minimal flows, namely, $\LO(\kappa)^\kappa$).
\end{remark}

\subsubsection*{Noncompact case}

The non-compact case is slightly more involved and we were unable to reach a full classification.

Let $H_{\alpha, k} = \Homeo(\omega^\alpha \cdot k)$ and $I_{\alpha, k, \beta} = \Homeo(\omega^\alpha \cdot k + \omega^\beta)$ be the homeomorphism groups of the non-compact ordinals. The same argument as above allows us to retrieve the parameters $\alpha$ and $k$:
\begin{itemize}
    \item $(k-1)!$ is the maximal cardinal of a finite discrete quotient of $H_{\alpha, k}$;
    \item $k!$ is the maximal cardinal of a finite discrete quotient of $I_{\alpha, k, \beta}$;
    \item $\chain(F_{\alpha, k}) = \alpha$;
    \item $\chain(I_{\alpha, k, \beta}) = \alpha$.
\end{itemize}

Consequently, there is no isomorphism between the groups of the family $H_{\alpha, k}$ if $k \geq 2$. But $0! = 1!$, hence a first impediment towards a classification:

\begin{question}
	Are there isomorphisms between $H_{\alpha, 1}$ and $H_{\alpha, 2}$?
\end{question}

More generally, by considering the one-point compactification, we can find isomorphisms between $H_{\alpha, k}$ and the fixator of a point of rank~$\alpha$  in $G_{\alpha, k}$, which is a subgroup of index $k$. In particular, $H_{\alpha, 1}$ is isomorphic to $G_{\alpha, 1}$. Are there other isomorphisms? Because of the parameters $\alpha$ and $k$, the question is of the following form (and encompasses the previous question):

\begin{question}
    Are there isomorphisms between $G_{\alpha, k}$ and $H_{\alpha, k + 1}$?
\end{question}

Similarly, $I_{\alpha, k, \beta}$ is isomorphic to the fixator of a point of rank~$\beta$ in $G_{\alpha, k}$, which is of infinite index. We were unfortunately unable to retrieve the parameter $\beta$ or to exclude isomorphisms with groups of the other families:

\begin{question}
    Are there isomorphisms between $I_{\alpha, k, \beta}$ and $I_{\alpha, k, \beta'}$? Are there isomorphisms between $I_{\alpha, k, \beta}$ and $G_{\alpha, k}$ or $H_{\alpha, k + 1}$?
\end{question}

\bibliographystyle{../../BIB/amsalpha}
\bibliography{../../biblio/biblio}

\end{document}